\documentclass[a4paper,11pt]{article}

\usepackage[english]{babel}
\usepackage{anysize}
\marginsize{3.5cm}{3.5cm}{2.5cm}{2.5cm}
\usepackage{fancyhdr}
\usepackage{amsfonts}
\usepackage{amsmath} 
\usepackage{tikz,graphicx,subfigure,overpic}
\usepackage{amssymb} 
\usepackage{amsthm}
\usepackage[colorlinks=true, linkcolor=blue, citecolor=blue]{hyperref}
\usepackage[fixlanguage]{babelbib}
\usepackage{comment}

\numberwithin{equation}{section}

\newtheorem{theorem}{Theorem}[section]

\newtheorem{corollary}[theorem]{Corollary}
\newtheorem{proposition}[theorem]{Proposition}
\theoremstyle{definition} 
\newtheorem{definition}[theorem]{Definition}
\newtheorem{assumption}[theorem]{Assumption}
\newtheorem{Remark}[theorem]{Remark}
\newenvironment{remark}{\begin{Remark}\rm}{\end{Remark}}
\newenvironment{example}{\begin{Example}\rm}{\end{Example}}
\newtheorem{Example}[theorem]{Example}

\newcommand{\eq}{\begin{equation}}
\newcommand{\qe}{\end{equation}}
\newcommand{\bs}{\boldsymbol}

\DeclareMathOperator{\diam}{diam}
\DeclareMathOperator{\supp}{supp}

\newcommand{\R}{\mathbb{R}}
\newcommand{\C}{\mathbb{C}}
\newcommand{\J}{\mathcal{J}}
\newcommand{\M}{\mathcal{M}}
\newcommand{\V}{\mathcal{V}}
\newcommand{\s}{\mathcal{S}}

\newcommand{\spac}{\mathcal{M}_{m_1}(\Delta_1)\times\cdots\times\mathcal{M}_{m_d}(\Delta_d)}

\title{Weakly Admissible \\Vector Equilibrium Problems}

\author{
\ Adrien Hardy 
\footnote{Institut de Math\'ematiques de Toulouse, Universit\'e de Toulouse, 31062 Toulouse, France.} \hspace{2 dd}\footnote{Department of Mathematics, Katholieke Universiteit Leuven, Celestijnenlaan 200 B,
3001 Leuven, Belgium. Email addresses: adrien.hardy@wis.kuleuven.be, arno.kuijlaars@wis.kuleuven.be} 
\,
and  
\ Arno B.J.  Kuijlaars 
\footnotemark[\value{footnote}] }

\begin{document}

\maketitle 

\begin{abstract}

We establish lower semi-continuity and strict convexity of the energy functionals for a
large class of vector equilibrium problems  in logarithmic potential theory. This, in particular,
implies the existence and uniqueness of a minimizer for such vector equilibrium problems.
Our work extends earlier results in that we allow unbounded supports without having strongly
confining external fields.  To deal with the possible noncompactness of supports, 
we map the vector equilibrium problem onto the Riemann sphere and our results follow
from a  study of vector equilibrium problems  on compacts in higher dimensions. 
Our results cover a number of cases that were recently considered in random matrix theory
and for which the existence of a minimizer was not clearly established yet.

\end{abstract}


\begin{comment}  % Adrien's abstract

We prove existence and uniqueness of minimizers for a large class of vector equilibrium problems. 
We actually show the associated energy functionals to have compact sub-level sets, and thus to be 
lower semi-continuous, in addition to be strict convex. 
This work provides an extension to the results present in the literature by considering vector 
equilibrium problems 
which involve measures with unbounded supports, but not necessarily confining external fields. 
The lack of compactness 
generated by such a weakening of the assumptions on vector equilibrium problems leads us to work 
on the Riemann sphere, and to introduce the wide class of weakly admissible vector equilibrium problems. 
This class concerns in 
particular many recent works related to random matrix theory, for which the existence of such minimizers was 
not clearly established yet.
\end{comment}

\section{Introduction}
A \emph{vector equilibrium problem} in logarithmic potential theory asks to find the minimizer 
of a functional  involving  logarithmic energies of measures lying in a prescribed set.  The origins of vector equilibrium problems lie in the works of Gonchar and Rakhmanov
on Hermite-Pad\'e approximation \cite{GR1, GR2, GRS}, where they are used to describe the limiting
distributions of the poles of the rational approximants \cite{NS}. More recently, vector equilibrium
problems also appeared in random 
models related to multiple orthogonal polynomials, such as random matrix ensembles, or non-intersecting diffusion 
processes;  see the surveys \cite{AK, K} 
and the references cited therein.

The question is to prove the existence and uniqueness of such minimizer. Results are already available in 
the literature \cite{BB, BKMW, NS, ST} but they do not cover yet a wider class of vector equilibrium problems 
arising from random matrix theory, among other things.  
Let us illustrate this by an example. In \cite{DK, DGK, DKM} the two matrix model, which is a model of two 
coupled random matrices, is investigated and the limiting mean eigenvalue distribution of one of the matrices is characterized 
in terms of the following vector equilibrium problem. Minimize  the energy functional
\begin{multline} \label{energy2MM}
 \iint \log\frac{1}{|x-y|}d\mu_1(x)d\mu_1(y) - \iint \log\frac{1}{|x-y|}d\mu_1(x)d\mu_2(y)  \\ 
  + \iint \log\frac{1}{|x-y|}d\mu_2(x)d\mu_2(y) - \iint \log\frac{1}{|x-y|}d\mu_2(x)d\mu_3(y)\\
  + \iint \log\frac{1}{|x-y|}d\mu_3(x)d\mu_3(y) +\int V_1(x)d\mu_1(x) + \int V_3(x)d\mu_3(x)
\end{multline}
over vectors of measures $(\mu_1,\mu_2,\mu_3)$ where $\mu_1$ and $\mu_3$ are measures on $\R$, $\mu_2$ is a 
measure on the imaginary axis $i\R$, and they have respective total masses $\|\mu_1\|=1$, $\|\mu_2\|=2/3$ 
and $\|\mu_3\|=1/3$. Moreover, $\mu_2$ is constrained by a measure $\sigma$ appearing in the problem, 
that is $\sigma-\mu_2$ has to be a (positive) measure. The external fields $V_1$ and $V_3$ in \eqref{energy2MM} 
are given continuous functions on $\mathbb R$ and $V_1$ has polynomial growth at infinity, while $V_3$ has compact support.

The existence of a unique minimizer $(\mu^*_1,\mu^*_2,\mu^*_3)$ plays a crucial role in the two matrix model 
investigation. Indeed, an important step for its asymptotic analysis is to normalize the associated Riemann-Hilbert problem at infinity, a procedure which is possible because of the existence of such a minimizer, and as a consequence the first component $\mu_1^*$ turns out to be the limiting mean eigenvalue distribution of one of the random matrices. Nevertheless, the proof of existence and uniqueness presented in \cite{DK,DKM} is rather complicated and moreover incomplete
since the lower semi-continuity of the energy functional \eqref{energy2MM} was implicitly 
assumed but not proved. 
There are other random models for which the existence of a unique minimizer for an associated vector equilibrium problem 
has not clearly been established and which will be covered by this work. 
Examples are non-intersecting squared Bessel 
paths models \cite{DKRZ, KMW} and a Hermitian random matrix model with an external source \cite{BDK}.

In the recent paper \cite{BKMW}, Beckermann et al.\ establish lower semi-continuity and existence of minimizers 
for vector equilibrium problems in situations more general than known before, 
but under an hypothesis of compactness (namely the presence of strongly confining external fields, in case
of unbounded sets) which is not present in the example \eqref{energy2MM}. It is the aim of this work to extend the methods
of \cite{BKMW} so as to cover the above examples. We restrict in this work to positive definite
interaction matrices, while the work \cite{BKMW} also includes semi-definite interaction matrices.

\section{Vector equilibrium problems on the complex plane}
We first introduce few definitions commonly used in  logarithmic potential theory. 
\subsection{Notions from potential theory}
For a measure $\mu$ on $\mathbb C$, the \emph{logarithmic energy} is defined by
\begin{equation}
\label{logenergy}
I(\mu) = \iint \log\frac{1}{|x-y|}d\mu(x)d\mu(y),
\end{equation}
and the \emph{logarithmic potential} at $x\in\mathbb C$ by
\begin{equation} 
\label{logpotential}
U^{\mu}(x)=\int\log\frac{1}{|x-y|}d\mu(y),
\end{equation}
whenever these integrals make sense. 
Here and in the following, by a \emph{measure} we always mean 
a positive finite Borel measure. Moreover, for two measures $\mu$ and $\nu$ on $\C$, their \emph{mutual energy} is given by
\begin{equation}
\label{mutualenergy}
I(\mu,\nu) = \iint \log\frac{1}{|x-y|}d\mu(x)d\nu(y),
\end{equation}
so that $I(\mu) = I(\mu,\mu)$.
These definitions are naturally extended to signed measures.

For a closed subset $\Delta \subset \mathbb C$ and a positive number $m > 0$, we use 
$\mathcal{M}_m(\Delta)$ to denote the set of measures $\mu$ having support $\supp(\mu)\subset\Delta$ and total mass $\|\mu\| = m$. 
Such a set $\mathcal{M}_m(\Delta)$ will always be equipped with its weak 
topology (i.e., the topology coming from duality with the Banach space of bounded continuous functions on $\Delta$). 
The Cartesian product $\spac$ of such sets carries the product topology. 

A closed subset $\Delta$ of $\C$ has \emph{positive capacity} if there exists a measure 
with support in $\Delta$  having finite logarithmic energy. 

\subsection{The class of weakly admissible vector equilibrium problems}
Let us now precise the assumptions for vector equilibrium problems concerned in this Section. Fix an integer $d\geq 1$.

\begin{assumption}\textbf{(Weak admissibility)}
\label{def1}\\
We make the following assumptions :
\begin{enumerate}
\item[(a)] 
$C=(c_{ij})$ is a $d\times d$ real symmetric positive definite matrix.
\item[(b)] 
$\bs \Delta = (\Delta_1,\ldots, \Delta_d)^t$ is a vector of closed subsets  of  $\C$ each having positive capacity.
\item[(c)]
$\bs V=(V_1, \ldots, V_d)^t$ is a vector of external fields  where each 
$V_i : \Delta_i\to  \R\cup\{+\infty\}$ is lower semi-continuous and finite on a set of positive capacity.
\item[(d)]
$\bs m=(m_1,\ldots,m_d)^t$ is a vector of positive numbers  such that
\begin{equation}
\label{weakadmis}
\liminf_{|x|\to \infty,\;x\in\Delta_i}\left(V_i(x)-\Big(\sum_{j=1}^dc_{ij}m_j\Big)\log(1+|x|^2)\right)>-\infty
\end{equation}
for every $i=1, \ldots , d$, provided $\Delta_i$ is unbounded. 
\end{enumerate}
\end{assumption}
\noindent
Given $C$, $\bs V$, $\bs \Delta$, $\bs m$ satisfying Assumption \ref{def1},
a  \emph{weakly admissible vector equilibrium problem} asks for minimizing the functional
\begin{equation}
\label{vep1}
\J_{\bs V}(\mu_1,\ldots,\mu_d)=\sum_{1\leq i,j \leq d }c_{ij}I(\mu_i,\mu_j)+\sum_{i=1}^d\int V_i(x)d\mu_i(x)
\end{equation}
over vectors of measures $(\mu_1,\ldots,\mu_d)$ lying in $\spac$, or in a subset thereof. 
The terminology \emph{weakly admissible} mainly refers to the growth conditions \eqref{weakadmis}, since it weakens all the growth assumptions presented in the literature, see also Remark \ref{eqpb} below. Indeed, it is assumed in \cite{NS} that the $\Delta_i$'s are compact sets, and both \cite{BKMW} and \cite[Section VIII]{ST} require for unbounded $\Delta_i$'s the stronger growth condition
\eq
\label{stronggrowthcond}
\lim_{|x|\rightarrow\infty,\;x\in\Delta_i}\frac{V_i(x)}{\log(1+|x|^2)}=+\infty, \qquad i\in\{1,\ldots,d\},
\qe
implying \eqref{weakadmis} for any $\bs m$.
%{\red The vector equilibrium problem may be called (strongly) admissible
%if the left-hand side of \eqref{weakadmis} is equal to $+\infty$ for every $i = 1, \ldots d$. 
%In that case the minimizers for \eqref{vep1} will have compact support. }

Moreover, note that there is no condition on the relative positions of the sets $\Delta_i$. They
could be disjoint (as assumed in \cite[Proposition V.4.1]{NS} and \cite[Theorem VIII.1.4]{ST} in case of attraction), 
but they could also overlap, even in case of attraction (i.e. $c_{ij} < 0$)
between the measures on $\Delta_i$ and $\Delta_j$. This feature is also present
in the work \cite{BKMW}.

\begin{example} \textbf{(Two matrix model)} \label{2matmodel}\\
The vector equilibrium problem for the functional \eqref{energy2MM} has the input data
\[
C= 
\begin{bmatrix} 1 & -1/2 & 0 \\
 -1/2 & 1 & -1/2\\
 0 & -1/2 & 1
\end{bmatrix},
\qquad \bs \Delta = \begin{pmatrix} \R \\ i\R \\ \R \end{pmatrix}, \quad 
\bs V = \begin{pmatrix} V_1 \\ 0 \\ V_3 \end{pmatrix}, \quad 
\bs m = \begin{pmatrix} 1 \\ 2/3 \\ 1/3 \end{pmatrix}
\]
which clearly satisfies  the conditions  (a), (b), and (c) of Assumption \ref{def1}. 
Since $C\bs m = \begin{pmatrix} 2/3 & 0 & 0 \end{pmatrix}^t$ we have
\[
	\bs V- C\bs m \log(1+|x|^2)= 
		\begin{pmatrix} V_1(x)-\frac{2}{3}\log(1+|x|^2) \\ 0 \\ V_3(x) \end{pmatrix}
\]
which means that condition (d) is satisfied as well, since there exists positive constants $c_1$, $c_2$ and $\alpha$ such that $V_1(x)\geq c_1|x|^\alpha-c_2$, and $V_3$ has a compact support. Thus the vector equilibrium problem is 
weakly admissible.
\end{example}

\begin{example} \textbf{(Banded Toeplitz matrices)} \label{Toeplitz} \\
A banded Toeplitz matrix $T_n$ with $p \geq 1$ upper and $q \geq 1$ lower diagonals has the form
\begin{equation}
   \big(T_n \big)_{jk}=a_{j-k},
    \qquad j,k = 1, \ldots, n,
\end{equation}
where $a_p a_{-q} \neq 0$ and $a_k = 0$ for $k \geq p+1$ and $k \leq -q - 1$. The limiting 
eigenvalue distribution of the matrices $T_n$ as the size $n$ tends to infinity
was characterized in \cite{DK0} by means of a vector equilibrium problem
with $d = p + q - 1$ measures without external fields $V_i$. 
The interaction matrix (which is tridiagonal) and the vector of masses are
\[ C = \begin{bmatrix}   1 & -\frac{1}{2} & 0 & \cdots & \cdots & 0 \\
 		- \frac{1}{2} & 1 & -\frac{1}{2} & &   & \vdots \\
 		0 & - \frac{1}{2} & 1 & \ddots & & \vdots \\ 
 		\vdots & & \ddots & \ddots & \ddots & 0 \\ 
 		\vdots & &  & \ddots & 1 & - \frac{1}{2} \\
 		0 & \cdots & \cdots & 0 & -\frac{1}{2} & 1 
\end{bmatrix}, \qquad \bs m = \begin{pmatrix} \tfrac{1}{q} \\ \vdots \\ \tfrac{q-1}{q} \\ 1 \\ 
	\tfrac{p-1}{p} \\ \vdots \\ \tfrac{1}{p} \end{pmatrix}. \] 
The sets $\Delta_i$ are curves in the complex plane, where $\Delta_q$ is compact but 
the others are unbounded. Note that all entries of $C \bs m$ are zero except for 
\[ \bigl( C \bs m \bigr)_q = 1 - \tfrac{q-1}{2q} - \tfrac{p-1}{2p} \geq 0. \]  
Since $\Delta_q$ is bounded,  the conditions of Assumption \ref{def1} are
satisfied even though the external fields are all absent. The corresponding vector equilibrium problem is weakly admissible. 

See \cite{Del,DD} for extensions to rational Toeplitz matrices and block Toeplitz matrices
which lead to a number of interesting variations on the above vector equilibrium problem.
\end{example}

\begin{remark} \textbf{(Scalar equilibrium problems)} 
\label{eqpb}\\
In the scalar case $d=1$ one may assume without loss of generality that $c_{11}=m_1 = 1$. 
Then the energy functional \eqref{vep1} with $V_1 = V$ and $\mu_1 = \mu$ reduces to
\[ \iint \log \frac{1}{|x-y|} \, d\mu(x) d\mu(y) + \int V(x) \, d\mu(x) \]
which differs from the one in \cite{ST} by a factor $2$ in the external field term. In 
the setting of \cite{ST} the external field is associated with the weight 
$w(x) = e^{-\frac{1}{2} V(x)}$, and then the equilibrium problem is called \emph{admissible} if
\[ \lim_{|x| \to \infty} |x| w(x) = 0, \]
which means that the left-hand side of \eqref{weakadmis} is equal to $+\infty$.
In \cite{Si} the scalar equilibrium problem is called \emph{weakly admissible} if
\[ \lim_{|x| \to \infty} |x| w(x) = \gamma > 0. \]
Observe that \eqref{weakadmis} is more general, since we do not require that
the limit of $V(x) - \log(1+ |x|^2)$ as $|x| \to \infty$ exists.
\end{remark}

\subsection{Extension of the energy functional definition}
\label{extension}
Note that  the energy functional \eqref{vep1} is not well defined for all measures 
since logarithmic energies may take the values $+\infty$ and $-\infty$ 
(the latter cannot happen for measures with compact support). 
One may restrict to measures satisfying the condition
\begin{equation}
\label{cond}
\qquad I(\mu)<+\infty \qquad \mbox{and} \qquad \int \log(1+|x|) \, d\mu(x)<+\infty
\end{equation}
so that \eqref{vep1} is always well defined, as it is done in \cite{BKMW, NS}. But it is also 
possible to extend naturally the definition of $\J_{\bs V}(\mu_1,\ldots,\mu_d)$ to situations 
where \eqref{cond} is not satisfied.  

We extend the energy functional \eqref{vep1} by mapping the vector equilibrium problem onto the Riemann sphere by means of 
inverse stereographic projection. Namely, let 
\begin{equation} \label{sphere}
 \s = \Big\{ (x_1, x_2, x_3) \in \R^3 \mid x_1^2 + x_2^2 + (x_3- \tfrac{1}{2})^2 = \tfrac{1}{4} \Big\}
\end{equation} 
be the sphere in $\R^3$ centered in $(0,0,1/2)$ with radius $1/2$ and $T:\C\cup\{\infty\}\to  \s$ the homeomorphism defined by 
\begin{equation} \label{mapT}
\begin{aligned}
	T(x)\,  & = \left(
		\frac{\mathrm{Re}(x)}{1+|x|^2}, \frac{\mathrm{Im}(x)}{1+|x|^2}, \frac{|x|^2}{1+|x|^2}\right), & x \in \C \\
	\end{aligned}
\end{equation}
and $T(\infty) = (0,0,1)$. Then, the following metric relation holds, see \cite[Lemma 3.4.2]{Ash},
\begin{equation}
\label{metric}
|T(x)-T(y)|=\frac{|x-y|}{\sqrt{1+|x|^2}\sqrt{1+|y|^2}}, \qquad x, y \in \C,
\end{equation}
where $| \cdot |$ denotes the Euclidean norm. 

For a measure $\mu$ on $\C$ we use $T_*\mu$  to denote its push forward by $T$, that is, 
$T_*\mu$ is the measure on $\s$ characterized by
\begin{equation}
\label{pushf}
\int f(s) \, dT_*\mu(s) = \int f\big(T(x)\big) \, d\mu(x)
\end{equation}
for every Borel function $f$ on $\s$. If $\mu$ and $\nu$ are two measures 
on $\C$ satisfying the condition \eqref{cond}, then \eqref{metric}, \eqref{pushf} easily yield
\begin{equation} 
\label{energyonsphere}
I\big(T_*\mu,T_*\nu\big) = I(\mu,\nu) + \frac{1}{2}\|\nu\|\int\log(1+|x|^2) \, d\mu(x) +\frac{1}{2}\|\mu\|\int\log(1+|x|^2)\, d\nu(x).
\end{equation}
As a consequence, we obtain for $\mu_i$'s which  satisfy  \eqref{cond}
\begin{equation}
\label{exten}
\J_{\bs V}(\mu_1,\ldots,\mu_d) = \sum_{1\leq i,j \leq d }c_{ij}I(T_*\mu_i,T_*\mu_j)+\sum_{i=1}^d\int \mathcal{V}_i(x) \, dT_*\mu_i(x)
\end{equation}
where the new external fields $\V_i : T(\Delta_i)\to \R\cup\{+\infty\}$ are defined by 
\begin{equation}
\label{defV1}
\V_i \big(T(x)\big) = V_i(x) - \Big(\sum_{j=1}^dc_{ij}m_j\Big)\log(1+|x|^2), \qquad x \in \Delta_i.
\end{equation}
The condition \eqref{weakadmis} thus states that the $\V_i$'s are bounded from below. 
In case $\Delta_i$ is unbounded, we extend the definition of $\V_i$ by putting
\begin{equation}
\label{defV2}
\V_i(0,0,1) = \liminf_{|x| \to \infty,\;x\in\Delta_i} \left(  V_i(x) - \Big(\sum_{j=1}^dc_{ij}m_j\Big)\log(1+|x|^2) \right).
\end{equation}
Then $\V_i$ is a lower semi-continuous function defined on a closed subset of $\s$.
Thus, \eqref{exten} motivates the following definition.

\begin{definition}
\label{def2}
We extend the definition of the energy functional  \eqref{vep1} to all vectors of measures in $\spac$ by setting
\begin{multline}
\label{vep3}
\J_{\bs V}(\mu_1,\ldots,\mu_d)
	=  \displaystyle \sum_{1\leq i,j \leq d }c_{ij}I(T_*\mu_i,T_*\mu_j)
		+ \sum_{i=1}^d\int \mathcal{V}_i(x) \, dT_*\mu_i(x)  \\
		\textrm{if } I(T_*\mu_i) < +\infty \textrm{ for every } i = 1, \ldots, d,
\end{multline}
where the $\V_i$'s are defined by \eqref{defV1} and \eqref{defV2}, and
\begin{equation} \label{vep4}
	\J_{\bs V}(\mu_1,\ldots,\mu_d) = + \infty \quad \textrm{otherwise.}
		\end{equation}		
\end{definition}

The main result of this work is the following.

\begin{theorem}
\label{th1}
Let $C$, $\bs\Delta$, $\bs V$ and $\bs m$ satisfy Assumption \ref{def1},
and let $\J_{\bs V}$ be the associated energy functional given by \eqref{vep3}, \eqref{vep4} in Definition \ref{def2}.
Then the following hold :
\begin{enumerate}
\item[{\rm (a)}] The sub-level set 
\begin{equation} \label{JVsublevel} 
	\Big\{ (\mu_1, \ldots, \mu_d) \in \spac \mid \J_{\bs V}(\mu_1, \ldots, \mu_d) \leq \alpha \Big\} 
\end{equation}
is compact for every $\alpha \in \mathbb R$. In particular $\J_{\bs V}$ is lower semi-continuous.

\item[{\rm (b)}] $\J_{\bs V}$ is strictly convex on the set where it is finite.

\end{enumerate}
\end{theorem}

The following is an immediate consequence of Theorem \ref{th1}.

\begin{corollary} \label{cor1}
 The functional $\J_{\bs V}$ admits a unique minimizer on $\spac$,
 as well as on any closed convex subset of $\spac$ that contains at least
 one element where $\J_{\bs V}$ is finite.
 \end{corollary}
The case of upper constraints is also covered by Corollary \ref{cor1}. Indeed, given any subset $J\subset\{1,\ldots,d\}$  
and (possibly unbounded) measures $(\sigma_j)_{j\in J}$, the subset of vectors of measures $(\mu_1, \ldots, \mu_d) \in \spac$ satisfying  $\mu_j \leq \sigma_j$
for  $j\in J$ is closed and convex. 

A question of interest is whether the component of such minimizer satisfy the condition \eqref{cond} or not. If the answer is affirmative, then by uniqueness the minimizer coincide with the one of \cite{BKMW}, at least when the  $V_i$'s  satisfy the strong growth condition \eqref{stronggrowthcond}. We relate this question to the regularity of logarithmic potentials, see Remark \ref{regularity}.

\begin{remark}\textbf{(Good rate function)} \\
Note that the condition (a) of Theorem \ref{th1} is what is necessary
to have a \emph{good rate function} in the theory of large deviations \cite{DZ}. More precisely Theorem \ref{th1} yields that 
\begin{equation}
	\label{ratefunction} 
		(\mu_1,\ldots,\mu_d) \mapsto \J_{\bs V}(\mu_1,\ldots,\mu_d)  - \min \J_{\bs V}
		\end{equation}
is a good rate function on $\spac$ as well as on every closed subset of $\spac$.

Whenever the minimizer of an energy functional $\J_{\bs V}$ describes the typical limiting
behavior in an interacting particle system,  it would be interesting to find out
if there is indeed a large deviation principle associated with it. 
Some results in this direction are obtained in \cite{ESS} for Angelesco ensembles, see also \cite{Bl}.
However for the energy functional \eqref{energy2MM} that is relevant for the eigenvalues 
of a random matrix in the two matrix model this remains an open problem.
\end{remark}

The extension of the definition for $\J_{\bs V}$ leads us to consider vector 
equilibrium problems on compact sets in higher dimensional spaces, for which we provide a general treatment 
in the next Section. Theorem \ref{th1} will appear as a consequence of this 
investigation, see Section \ref{comebacktoC}.

\section{Vector equilibrium problems on compacts in $\R^n$}
\label{VEPonRn}
In this section, let $d,n\geq 1$ and $K\subset\R^n$ be a compact set with positive capacity. We now provide a general treatment for vector equilibrium problems involving $d$ measures on $K$. 

We first consider in Section \ref{intro}  vector equilibrium problems involving  measures with unit mass  and no external field, for which we claim lower semi-continuity and strict convexity, see Theorem \ref{th2}. We then show how such result easily extends to vector equilibrium problems with general masses and external fields, see Theorem \ref{th3}. The proof of Theorem \ref{th2}, which is the main part of  Section \ref{VEPonRn}, is given in Section \ref{proof}. Finally, we come back to weakly admissible vector equilibrium problems on $\C$ and provide a proof for Theorem \ref{th1} in Section \ref{comebacktoC}, as a corollary of Theorem \ref{th3}.

\subsection{Introduction}
\label{intro}
For measures on $K$, we again use the definitions 
\eqref{logenergy}--\eqref{mutualenergy} where $| \cdot |$ stands for the Euclidean norm. 
This notation was already used in \eqref{energyonsphere} for measures on the sphere $\s \subset \R^3$.

The following result is a consequence of  \cite[Theorem 2.5]{CKL}.
\begin{proposition} 
\label{posdef}
Let $\mu$ and $\nu$ be measures on  $K$ having finite logarithmic energy 
and  same total mass $\|\mu\|=\|\nu\|$. Then $I(\mu-\nu) \geq 0$ with equality if and only if $\mu=\nu$.
\end{proposition}

As a consequence of Proposition \ref{posdef} and of the fact 
that $K$ has finite diameter,  we obtain for any measures $\mu$ and $\nu$ supported in $K$ having finite logarithmic energy that $I(\mu,\nu)$ is finite. Indeed one can assume $\|\mu\|=\|\nu\|=1$ without loss of generality and then we have 
\begin{equation*} \label{Imunulower} 
	I(\mu, \nu) = \iint \log \frac{1}{|x-y|} \, d\mu(x) d\nu(y) \geq  \log \frac{1}{\diam K} > - \infty.
	\end{equation*}
Moreover by Proposition \ref{posdef}
\begin{equation*} \label{Imunuupper}
	\begin{aligned}
	I(\mu, \nu) & = 
	\frac{1}{2}
	\big(I(\mu)+I(\nu)-I(\mu-\nu)\big) \leq \frac{1}{2} \big( I(\mu) +  I(\nu) \big) 	<+\infty.
		\end{aligned}
\end{equation*} 
Given a $d\times d$ symmetric positive definite matrix $C=(c_{ij})$, 
we consider the quadratic map defined for vectors of measures $(\mu_1,\ldots,\mu_d)$ on $K$ by
\begin{equation}
\label{J0}
J_0(\mu_1,\ldots,\mu_d)=  \begin{cases} 
	\displaystyle \sum_{1\leq i, j \leq d}c_{ij} I(\mu_i,\mu_j) & \mbox{if all  $I(\mu_i)<+\infty$,}\\
	+\infty & \mbox{otherwise}. \end{cases}
\end{equation}

The central result of this section is the following.

\begin{theorem}
\label{th2}
For any $d\times d$ symmetric positive definite matrix $C$, the functional $J_0$ defined in \eqref{J0} is lower semi-continuous on $\M_1(K)^d$ and strictly 
convex on the set where it is finite.
\end{theorem}

The proof of Theorem \ref{th2} is given in Section \ref{proof}. We first show how Theorem \ref{th2} applies to vector equilibrium problems with external fields on $K$ with the following data.

\begin{assumption}\
\label{dataK}
\begin{enumerate}
\item[(a)] 
$C=(c_{ij})$ is a $d\times d$ real symmetric positive definite matrix.
\item[(b)]
$\bs V=(V_1, \ldots, V_d)^t$ is a vector of external fields  where each 
$V_i : \Delta_i\to  \R\cup\{+\infty\}$ is lower semi-continuous and finite on a set of positive capacity.
\item[(c)]
$\bs m=(m_1,\ldots,m_d)^t$ is a vector of positive numbers. 
\end{enumerate}
\end{assumption}
\noindent
A vector equilibrium problem asks to minimize the following energy functional  
\begin{equation}
\label{JV}
J_{\bs V}(\mu_1,\ldots,\mu_d)=J_0(\mu_1,\ldots,\mu_d)+\sum_{i=1}^d\int V_i(x)d\mu_i(x),
\end{equation}
where $J_0$ is as in \eqref{J0}, over vectors of measures $(\mu_1,\ldots,\mu_d)$ lying in $\M_{m_1}(K)\times\cdots\times \M_{m_d}(K)$ (or in some closed convex subset thereof). A consequence of Theorem \ref{th2} is the following.

\begin{theorem} 
\label{th3}
If   $C$, $\bs V$ and $\bs m$ satisfy Assumption \ref{dataK}, then the functional $J_{\bs V}$ defined in \eqref{JV} is lower semi-continuous on the compact set 
$\M_{m_1}(K)\times\cdots\times\M_{m_d}(K)$ and strictly convex on the set where it is finite. Thus $J_{\bs V}$ admits a unique minimizer on $\M_{m_1}(K)\times\cdots\times\M_{m_d}(K)$,
as well as on every closed convex subset of $\M_{m_1}(K)\times\cdots\times\M_{m_d}(K)$ that contains at least one element where $J_{\bs V}$ is finite.
\end{theorem}

\begin{proof}[Proof of Theorem \ref{th3}.] 
Since $V_i$ is lower semi-continuous, there exists an increasing sequence $(V_i^M)_{M}$ of 
continuous functions on $K$ such that $\sup_{M}V^M_i=V_i$. 
By monotone convergence, the map
\[
\mu\mapsto \int V_i(x)d\mu(x) = \sup_{M}\int V^M_i(x)d\mu(x)
\]
is  lower semi-continuous on $\M_{1}(K)$, being the supremum of a family of continuous maps, 
and so is the linear map
\begin{equation} \label{extfieldsmap}
	(\mu_1,\ldots,\mu_d)\mapsto \sum_{i=1}^d\int V_i(x)d\mu_i(x)
\end{equation}
on $\M_{1}(K)^d$. Thus $J_{\bs V}$ is lower semi-continuous on $\M_1(K)^d$ by Theorem \ref{th2}.
Since  \eqref{extfieldsmap} is a linear map in the $\mu_i$'s which is bounded from below, we also find from Theorem \ref{th2} that $J_{\bs V}$ 
is strictly convex on the part of $\M_1(K)^d$ where it is finite, which proves the theorem in case all $m_i = 1$.

For the case of general masses $m_i > 0$, we note that if $\mu_i = m_i \nu_i$ for $i =1, \ldots, d$, then
\begin{equation} \label{generalmasses} J_{\bs V}(\mu_1,\ldots,\mu_d) = \sum_{1\leq i , j \leq d} c_{ij} m_i m_j I(\nu_i, \nu_j)
	+ \sum_{i =1}^d m_i \int V_i(x) \, d\nu_i(x). 
	\end{equation}
	The matrix $(c_{ij} m_i m_j)_{i,j=1}^d$ is symmetric positive definite
which implies by what we just proved that the right-hand side of \eqref{generalmasses}
is lower semi-continuous on $\M_1(K)^d$ and strictly convex on the set where it is finite.
Then the same holds for the left-hand side seen as a functional on $\spac$, and
Theorem \ref{th3} follows.
\end{proof}

In the next subsection we prove Theorem \ref{th2}.

\subsection{Proof of Theorem \ref{th2}}
\label{proof}

For $\bs \mu = (\mu_1,\ldots,\mu_d)\in\M_1(K)^d$, we also write $J_0(\bs \mu) = J_0(\mu_1,\ldots,\mu_d)$ 
 for convenience. 
 
\paragraph{Proof of strict convexity}

Being a positive definite matrix, we note that $C$ admits a Cholesky decomposition
\begin{equation}
\label{cholesky}
	C = B^t B  
\end{equation}
where $B =(b_{ij})$ is upper triangular and $b_{ii}>0$ for $i=1, \ldots, d$. 
The factorization \eqref{cholesky} implies that 
\begin{equation} \label{J0cholesky} 
	J_0(\bs \mu) = \sum_{i=1}^d I \left( \sum_{j=1}^d b_{ij} \mu_j \right)
	\end{equation}
	whenever $\mu_1, \ldots, \mu_d$ have finite logarithmic energy.

We prove the following statement, which is similar to \cite[Proposition 2.8]{BKMW}.
\begin{proposition}  \
\label{posdefJ}
\begin{enumerate}
\item[{\rm (a)}]
Let $\bs\mu=(\mu_1,\ldots,\mu_d)$, $\bs \nu =(\nu_1,\ldots,\nu_d)$  be vectors of probability measures on 
$K$ having  finite logarithmic energy. Then $J_0(\bs\mu -\bs\nu)\geq 0$ with equality if and only if $\bs\mu=\bs\nu$.
\item[{\rm (b)}]
$J_0$ is strictly convex on $ \big\{\bs\mu\in\M_1(K)^d \mid J_0(\bs\mu)<+\infty\big\}$.
\end{enumerate}
\end{proposition}

\begin{proof} 
(a) The Cholesky decomposition $C=B^tB$ yields (similar to \eqref{J0cholesky}) 
\begin{equation} \label{J0posdef}
	J_0(\bs\mu -\bs\nu) = \sum_{i=1}^d I\Big(\sum_{j=1}^d b_{ij}(\mu_j-\nu_j)\Big) 
\end{equation}
and, since for any $1\leq i\leq d$ the signed measure $\sum_{j=1}^db_{ij}(\mu_j-\nu_j)$ has total mass zero, each term in the right-hand side of \eqref{J0posdef} is non-negative by Proposition \ref{posdef}.
Thus $J_0(\bs\mu -\bs\nu) \geq 0$. Equality holds if and only if $\sum_{j=1}^d b_{ij}(\mu_j-\nu_j) = 0$
for every $i = 1, \ldots, d$, and this means that $\bs\mu=\bs\nu$ since $B$ is invertible.

(b) Let $\boldsymbol\mu, \boldsymbol\nu \in\M_1(K)^d$ 
satisfy  $J_0(\bs \mu), J_0(\bs \nu)<+\infty$. Then each component has finite logarithmic energy,
and we obtain  by bilinearity of $(\mu,\nu)\mapsto I(\mu,\nu)$ that
\[
	J_0\big(t \bs\mu + (1-t)\bs\nu \big) = 
		t J_0(\bs\mu)+(1- t) J_0(\bs\nu) - t(1-t) J_0(\bs\mu-\bs\nu)
\] 
for every $0 < t < 1$. Then part (b) follows from part (a).
\end{proof}

\paragraph{Proof of lower semi-continuity}

The next proposition is the main step in establishing lower semi-continuity of $J_0$
at the points where it is infinite.  The proof is inspired from the one of \cite[Proposition 2.11]{BKMW}.

\begin{proposition}
\label{prooflsc1}
Let $(\bs\mu^N)_N=\big((\mu^N_1,\ldots,\mu_d^N)\big)_N$ be a sequence in $\M_1(K)^d$ satisfying $J_0(\bs\mu^N)<+\infty$ for all $N$. Assume there exists $k \in\{1,\ldots,d\}$ such that 
\[ \lim_{N \to \infty}  I(\mu_{k}^N)=+\infty. \]
Then 
\[ \lim_{N \to \infty} J_0(\bs\mu^N) =+\infty. \]
\end{proposition}

\begin{proof}
We may assume $k=d$ without loss of generality. By \eqref{J0cholesky} and the fact that
$B$ is upper triangular, we have for every $N$,
\begin{equation}
\label{A1}
	J_0(\bs\mu^N) = 
	\sum_{i=1}^{d-1}I\Big(\sum_{j=1}^db_{ij}\mu^N_j\Big)+ b_{dd}^2 I(\mu^N_d).
\end{equation}

Note that the map $\mu\mapsto I(\mu)$ is lower semi-continuous on $\M_1(K)$.
For compact $K \subset \mathbb R^2 \simeq \C$ this is proved in \cite[Chapter 5, Theorem 2.1]{NS} for example,
but the proof applies without any modification to higher dimensions.
Thus by Proposition \ref{posdefJ} (b) it  has a unique minimizer $\omega$ on $\M_1(K)$. 
One can moreover show that this minimizer has constant logarithmic potential $U^{\omega}$  on $K$ 
up to a set $E$ of zero capacity (see \cite[Theorem I.1.3 and Remark I.1.6]{ST}), 
and that $\mu(E)=0$ for any measure $\mu$ on $K$ having finite logarithmic energy (see \cite[Remark I.1.7]{ST}). 

Then, Proposition \ref{posdef} yields for $i = 1, \ldots, d$, 
\[
I\Big(\sum_{j=1}^db_{ij}\mu^N_j -\big(\sum_{j=1}^db_{ij}\big)\omega \Big)\geq 0,
\] 
which implies for $i=1, \ldots, d-1$ that
\begin{align}
\label{A2}
I\Big(\sum_{j=1}^db_{ij}\mu^N_j\Big) 
&  \geq  2\big(\sum_{j=1}^db_{ij}\big) I\Big(\sum_{j=1}^db_{ij}\mu^N_j,\omega\Big) - \big(\sum_{j=1}^db_{ij}\big)^2 I(\omega).
\end{align}
Since $U^{\omega} = \rho$ is constant on $K\setminus E$, it easily follows that 
$I(\omega) = \rho$ and   
\[   I\Big(\sum_{j=1}^db_{ij}\mu^N_j,\omega\Big) = \sum_{j=1}^d b_{ij} \int U^{\omega}(x) \, d\mu_j^N(x) = 
	     \Big( \sum_{j=1}^d b_{ij}  \Big) I(\omega), \]
where the last equality holds since $U^{\omega} = I(\omega)$ on $K \setminus E$ and $\mu_j^N(E) = 0$ for every $j=1, \ldots, d$.
Using this in \eqref{A2} we find
\begin{align}
\label{A3}
I\Big(\sum_{j=1}^db_{ij}\mu^N_j\Big)  \geq  \Big(\sum_{j=1}^db_{ij}\Big)^2 I(\omega).
\end{align}
Summing \eqref{A3} over $i=1,\ldots, d-1$ and using \eqref{A1}, we find that 
\[ J_0(\bs\mu^N) \geq  \sum_{i=1}^{d-1} \Big(\sum_{j=1}^d b_{ij} \Big)^2 I(\omega) + b_{dd}^2 I(\mu_d^N). \] 
Thus if $I(\mu_d^N) \to +\infty$ as $N \to \infty$ we also have $J_0(\bs \mu^N) \to +\infty$
which completes the proof of Proposition \ref{prooflsc1}.
\end{proof}

The next proposition deals with lower semi-continuity of $J_0$ at
points where it is finite. We follow the lines of the proof of \cite[Proposition 2.9]{BKMW} and simplify it by considering a different way to approximate measures.

\begin{proposition}
\label{prooflsc2}
Let $(\bs\mu^N)_N=\big((\mu^N_1,\ldots,\mu_d^N)\big)_N$ be a sequence in $\M_1(K)^d$ satisfying $J_0(\bs\mu^N)<+\infty$ for all $N$. Assume $(\bs\mu^N)_N$ converges towards a limit  $\bs\mu=(\mu_1,\ldots,\mu_d)$ with $J_0(\bs\mu)<+\infty$.
 Then 
\begin{equation}
\label{B0}
\liminf_{N\to \infty} J_0(\bs\mu^N)\geq J_0(\bs\mu).
\end{equation}
\end{proposition}

\begin{proof}
We embed $\R^n$ into $\R^{n+1}$ in the obvious way, namely if $(e_1, e_2, \ldots, e_{n+1})$ 
is the standard orthonormal basis of $\R^{n+1}$ then we identify $\R^n$ with the linear span of $e_1, \ldots, e_n$.
In this way we also consider $K \subset \R^n$ as a subset of $\R^{n+1}$.

For $r>0$, let $\delta_r$ be the Dirac measure at the point $r e_{n+1} = (0, 0, \ldots, 0, r) \in \R^{n+1}$. 
For a measure $\mu$ on $\R^n$ we then have that  the convolution  $\mu*\delta_r$  yields 
a measure on $\R^{n+1}$ which is the translation  of $\mu$ along $r e_{n+1}$. 
Then for each $N$, the quantity $J_0(\bs\mu^N-\bs\mu^N*\delta_r)$, where the convolution is taken componentwise,  
makes sense and is non-negative by Proposition \ref{posdefJ} (a).  As a consequence we have
\begin{equation}
\label{B1}
J_0(\bs\mu^N)+J_0(\bs\mu^N*\delta_r)\geq 2\sum_{1 \leq i,j \leq d} c_{ij}I\big(\mu_i^N*\delta_r, \mu^N_j \big).
\end{equation}
Since the convolution with $\delta_r$ is just a translation and the logarithmic kernel $\log \frac{1}{|x-y|}$
is translation invariant, the two terms on the left-hand side of \eqref{B1} are the same.
We thus obtain from \eqref{B1}
\begin{equation}
\label{B2}
J_0(\bs\mu^N) \geq  \sum_{1 \leq i,j \leq d} c_{ij}I\big(\mu_i^N*\delta_r, \mu^N_j\big).
\end{equation}

Next, we compute by using orthogonality between elements of $\R^n$ and $e_{n+1}$  that
\begin{align*}
I\big(\mu_i^N*\delta_r,\mu^N_j\big) &  = \iint \log\frac{1}{|x-y|}d\big(\mu^N_i*\delta_r \big)(x)d\mu^N_j(y)\\
 & = \iint \log\frac{1}{|x-y+ r e_{n+1}|}\,d\mu^N_i(x)d\mu^N_j(y)\\
 & = \iint \log\frac{1}{\sqrt{|x-y|^2+r^2}}\,d\mu^N_i(x)d\mu^N_j(y).
\end{align*}
Since for fixed $r>0$ the map $(x,y)\mapsto \log (1/ \sqrt{|x-y|^2+r^2})$ is continuous on 
$K \times K$ and $(\bs \mu^N)_N$ converges towards $\bs \mu$, we obtain 
\[ \lim_{N \to \infty} I \big( \mu_i^N * \delta_r, \mu^N_j \big) =
	 \iint \log \frac{1}{\sqrt{|x-y|^2 + r^2}} \, d\mu_i(x) d\mu_j(y), \]
for every $i,j = 1, \ldots, d$, so that by \eqref{B2},
\begin{equation}
\label{B3}
\liminf_{N\to \infty}J_0(\bs\mu^N)\geq \sum_{1 \leq i,j \leq d} c_{ij}
	 \iint \log \frac{1}{\sqrt{|x-y|^2 + r^2}} \, d\mu_i(x) d\mu_j(y).
\end{equation}
The inequality \eqref{B3} holds for every $r > 0$. For every $x,y\in K$ and $0 < r \leq 1$, we have the inequalities 
\[
	\frac{1}{2} \log \frac{1}{(\diam K)^2 + 1} \leq \log\frac{1}{\sqrt{|x-y|^2+r^2}}\leq \log\frac{1}{|x-y|}.
\] 
Thus, since the $\mu_i$'s have finite logarithmic energy by assumption, we obtain by dominated convergence
\begin{equation}
\label{B4} 
	\lim_{r\to \,0}\iint \log \frac{1}{\sqrt{|x-y|^2 + r^2}} \, d\mu_i(x) d\mu_j(y)
	=I(\mu_i,\mu_j) .
	\end{equation}
Letting $r \to 0$ in \eqref{B3} and using \eqref{B4} we obtain \eqref{B0}.
\end{proof}

\begin{proposition} 
\label{LSC}
$J_0$ is lower semi-continuous on $\M_1(K)^d$.
\end{proposition}

\begin{proof}
Suppose $(\bs \mu^N)_N$ is a sequence in $\M_1(K)^d$ that converges
to $\bs \mu $. In order to prove that
\begin{equation}
\label{B5}
\liminf_{N\to\infty} J_0(\bs\mu^N)\geq J_0(\bs\mu)
\end{equation}
we may assume that $J_0(\bs \mu^N) < +\infty$ for
every $N$. If $J_0(\bs \mu) < +\infty$, then \eqref{B5} follows
from Proposition \ref{prooflsc2}. If $J_0(\bs \mu) = +\infty$
then by the definition \eqref{J0} we must have $I(\mu_k) = +\infty$
for at least one $k \in \{1, \ldots, d\}$. By lower semi-continuity of $\mu\mapsto I(\mu)$ on $M_1(K)$ it then follows that
\[ \lim_{N \to \infty} I(\mu_k^N) = +\infty, \]
and  \eqref{B5} follows from Proposition \ref{prooflsc1}.
\end{proof}

The proof of Theorem \ref{th2} is therefore complete. 

\subsection{Proof of Theorem \ref{th1}}
\label{comebacktoC}

Equipped with Theorem \ref{th3}, it is now easy to provide a proof for Theorem \ref{th1} as announced in Section \ref{extension}.

\begin{proof}[Proof of Theorem \ref{th1}.]
Given $C$, $\bs\Delta$, $\bs V$ and $\bs m$ satisfying Assumption \ref{def1}, introduce the vector of external fields $\bs\V=(\V_1,\ldots,\V_d)^t$ where $\V_i :\s\to\R\cup\{+\infty\}$ is defined in the following way. On $T(\Delta_i)$ define $\V_i$ from $V_i$ as in \eqref{defV1} and, if $\Delta_i$ is unbounded, extend the definition of $\V_i$ to $(0,0,1)$ using \eqref{defV2}. Then set $\V_i=+\infty$ elsewhere. Each $\V_i$ is then lower semi-continuous and finite on a set of positive capacity.  

(a)  By construction  the relation 
\begin{equation}
\label{C1}
\J_{\bs V}(\mu_1,\ldots,\mu_d)=J_{\bs \V}(T_*\mu_1,\ldots,T_*\mu_d)
\end{equation}
holds for all $(\mu_1,\ldots,\mu_d) \in\spac$, see Definition \ref{def2} and \eqref{J0}--\eqref{JV}.  As a consequence, we have 
for all $\alpha\in\R$ the relation between the sub-level sets  of $\J_{\bs V}$ and $J_{\bs\V}$
\begin{multline}
\label{C2}
T_*\times\cdots\times T_*\left(\Big\{ \bs\mu\in \spac \mid \J_{\bs V}(\bs\mu) \leq \alpha \Big\}\right)  \\
=  \Big\{ \bs\mu\in \M_{m_1}(\s)\times\cdots\times\M_{m_d}(\s) \mid J_{\bs \V}(\bs\mu) \leq \alpha \Big\}.
\end{multline}
Now we use Theorem \ref{th3} with  $C$, $\bs\V$, $\bs m$, which satisfy Assumption \ref{dataK}, and $K=\s\subset\R^3$. 
The theorem gives that $J_{\bs\V}$ has compact sub-level sets since $J_{\bs \V}$ is 
lower semi-continuous on the compact $\M_{m_1}(\s)\times\cdots\times\M_{m_d}(\s)$.  Since $T$ is an 
homeomorphism from $\C$ to $\s\setminus\{(0,0,1)\}$, it is not hard to check that $T_*$ is an homeomorphism 
from $\M_1(\C)$ to $\{\mu\in\M_1(\s) \mid \mu(\{(0,0,1)\})=0\}$,  so that part (a) follows from \eqref{C2} because a measure having a Dirac mass at $(0,0,1)$ has necessarily infinite logarithmic energy.

(b) Theorem \ref{th3} moreover yields that $J_{\bs\V}$ is strictly convex where it is finite. This clearly 
implies part (b) from  \eqref{C1} since $T_*$ is a linear injection. 
\end{proof}

\begin{remark}
\label{regularity}
A priori, the minimizer $(\mu_1,\ldots,\mu_d)$ of $\J_{\bs V}$ provided by Corollary \ref{cor1} could be such that  
\begin{equation}
\label{tail}
\int \log(1+|x|)\, d\mu_i(x)=+\infty \qquad \text{for some $i \in \{1, \ldots, d\}$}.
\end{equation}
In fact \eqref{tail} can only happen if the logarithmic potential $U^{T_*\mu_i}$ is infinite 
at the north pole of $\s$. Indeed, letting $y\to\infty$  in \eqref{metric}, we obtain for any $x\in\C$
\[
|T(x)-(0,0,1)|=\frac{1}{\sqrt{1+|x|^2}}
\] 
and thus we obtain from \eqref{pushf} the following equivalence  for any measure $\mu$ on $\C$
\[
\int \log(1+|x|)d\mu(x)=+\infty  \; \iff \;   U^{T_* \mu}(0,0,1) = +\infty.
\]
\end{remark}

\subsection*{Acknowledgments}
We would like to thank Bernhard Beckermann and Ana Matos for 
stimulating discussions on their paper \cite{BKMW} which had a big influence on
this work.

The authors are supported by FWO-Flanders projects G.0427.09 and  the Belgian Interuniversity
Attraction Pole P06/02.

The second author is also supported by FWO-Flanders project G.0641.11, 
 K.U.~Leuven research grant OT/08/33, and  research grant  MTM2011-28952-C02-01
from the Ministry of Science and Innovation of Spain and the European Regional Development Fund (ERDF).

\end{document}